\documentclass[10pt]{amsart}
\usepackage{graphicx}
\usepackage{latexsym}
\usepackage{fancyhdr}
\usepackage{amsmath, amssymb}
 \usepackage[utf8]{inputenc}
\usepackage[all]{xy}
\usepackage{pdflscape}
\usepackage{longtable}
\usepackage{rotating}
\usepackage{option_keys}
\usepackage[polish, english]{babel}
\usepackage{lmodern}
\usepackage[T1]{fontenc}
\usepackage{soul}
\usepackage{subfigure}
\usepackage{mathrsfs}
\usepackage{stmaryrd}
\usepackage{tensor}
\usepackage{enumerate}
\usepackage{hyperref}

\usepackage{color}

\newcommand{\lgw}{\longrightarrow}
\newcommand{\lgm}{\longmapsto}
\newcommand{\Supp}{\operatorname{Supp}}
\newcommand{\si}{\sigma}
\newcommand{\lb}{\llbracket}
\newcommand{\rb}{\rrbracket}
\newcommand{\ovl}{\overline}

\newcommand{\ord}{\operatorname{ord}}

\newcommand{\wdh}{\widehat}

\newcommand{\la}{\lambda}
\renewcommand{\O}{\mathcal{O}}

\renewcommand{\L}{\mathbb{L}}

\newcommand{\Z}{\mathbb{Z}}

\renewcommand{\k}{\Bbbk}
\newcommand{\R}{\mathbb{R}}
\newcommand{\K}{\mathbb{K}}

\newcommand{\N}{\mathbb{N}}
\newcommand{\C}{\mathbb{C}}

\newcommand{\Q}{\mathbb{Q}}

\newcommand{\lcm}{\operatorname{lcm}}
\renewcommand{\a}{\alpha}

\renewcommand{\b}{\beta}
\newcommand{\g}{\gamma}

\renewcommand{\phi}{\varphi}

\renewcommand{\o}{\omega}

\renewcommand{\P}{\mathbb{P}}
\newcommand{\ini}{\operatorname{in}}

\newcommand{\fract}[2]{\hbox{\leavevmode
  \kern.1em \raise .25ex \hbox{\the\scriptfont0 $#1$}\kern-.1em }\big/
  {\hbox{\kern-.15em \lower .5ex \hbox{\the\scriptfont0 $#2$}} }}

\theoremstyle{plain}
\newtheorem{theorem}{Theorem}[section]
\newtheorem{lemma}[theorem]{Lemma}
\newtheorem{corollary}[theorem]{Corollary}

\theoremstyle{definition}
\newtheorem{definition}[theorem]{Definition}
\newtheorem{example}[theorem]{Example}
\newtheorem{remark}[theorem]{Remark}
\newtheorem{question}[theorem]{Question}

\begin{document}

\title{About Eisenstein's Theorem}

\author{Guillaume Rond}




\keywords{algebraic power series}


\begin{abstract}
The aim of this survey papier is to present a result due to Eisenstein, to prove a generalized version of it, and to present some applications of this Eisenstein's Theorem, in particular to the study of the algebraic closure of the field of power series in several indeterminates.
\end{abstract}

\thanks{
This research was funded, in whole or in part, by lÕAgence Nationale de la Recherche (ANR), project ANR-22-CE40-0014. For the purpose of open access, the author has applied a CC-BY public copyright licence to any Author Accepted Manuscript (AAM) version arising from this submission.
}

\maketitle
\emph{Dedicated to Bernard Teissier who has always been a source of inspiration}\\
\section{Introduction}
On July 8, 1852, Gotthold  Eisenstein wrote a two pages long note \cite{Eis} whose aim is to state the following result:

\begin{theorem}[Eisenstein's Theorem]\label{thm:Eis}
Let $f\in\Q\lb x\rb$ be an algebraic power series where $x$ is a single indeterminate. Then there is an integer $a\in\N^*$ such that $af(ax)\in\Z\lb x\rb$.
\end{theorem}
We recall that an \emph{algebraic power series} is a formal power series $f$ for which there exists a nonzero polynomial $P(x,y)\in\Q[x,y]$ such that $P(x,f)=0$. If such a series is written as 
$f=\sum_\ell f_\ell x^\ell$, the conclusion of the  theorem can be reformulated as 
$$\forall \ell\in\N,\ a^{\ell+1}f_\ell\in\Z.$$
This result is an arithmetic tool that can be used to prove that some power series are transcendental. The study of algebraicity of functions defined by power series is an old problem that goes back, at least, to the work of Newton. This is related to the study of the transcendence of values of such functions; the algebraicity of a power series is also an indication that its coefficients satisfy strong relations and may be not to difficult to compute.

Eisenstein, like Galois and Abel, had a short life. He had a difficult childhood marked by illness: he was the only one of his siblings to survive meningitis. Though he survived it, he always had fragile health, being frequently ill, which led him to become hypochondriac. As a result, he always worked with a sense of urgency, with the idea that his time was limited, as explained by A. Weil  \cite[p. 4]{Wei}: \\
"As any reader of Eisenstein must realize, he felt hard pressed for time during the whole of his short mathematical career. As a young man he complains of nervous ailments which often compel him to interrupt his work; later, he developed tuberculosis, and died of it in 1852 at the age of 29. His papers, although brilliantly conceived, must have been written by fits and starts, with the details worked out only as the occasion arose; sometimes a development is cut short, only to be taken up again at a later stage. Occasionally Crelle let him send part of a paper to the press before the whole was finished. One is frequently reminded of Galois' tragic remark "Je n'ai pas Ie temps"."\\
This explains that the  aforementioned note is very short and contains no proof. Eisenstein starts by remarking that, when we expand the root square of $1+x$, the coefficient of $x^\ell$ is a rational number whose denominator is $2^k$ where $k\leq 2\ell$. Then he explains that this is general fact: he states essentially the content of Theorem \ref{thm:Eis}. His style is very particular as he does not use any mathematical symbol for this, but takes one page to explain this statement. Then he ends the note by  considering the  logarithm function: this one is transcendental since the sequence of denominators of the coefficients of its Taylor expansion is the sequence of natural numbers. He also adds that the same remark applies to $e^x$:
\begin{figure}[h]
  \centering
  \includegraphics[scale=0.715]{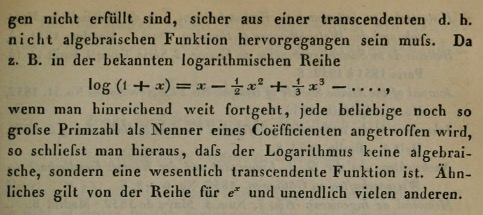}
\end{figure}\\

Eisenstein died of tuberculosis three months after the publication of this note,  October the 11th.
This is Eduard Heine who wrote a proof one month later, in November \cite{Hei}. His paper starts by reproducing the integrality of the note of Eisenstein. Then he explains that once this property of algebraic series has been discovered, the proof is not difficult to give, thus giving credit to Eisenstein for this result, which he calls Eisenstein's theorem ("der Eisensteinsche Satz"). 

In this paper we provide a general version of Eisenstein's Theorem in several indeterminates over a domain of any characteristic. Then we review some questions related to Eisenstein's Theorem. In particular we show how to apply this theorem to answer some questions about the algebraic closure of the field of power series, following \cite{Ron1}. In particular we show how to partially answer a question of Bernard Teissier about Jung-Abhyankar Theorem.

I thank Tanguy Rivoal and Julien Roques for indicating me the references \cite{Cla} and \cite{Saf}.

\section{Eisenstein's Theorem for generalized power series}
Now we present we prove a general version of Eisenstein's Theorem for Hahn power series with no accumulation points in their supports. First of all we start by introducing some terminology:

\begin{definition}
Let   $A$ be a commutative ring and $x$ be a single indeterminate. The set of polynomials $A[x^{\R_+}]$ is the set of finite sums $\sum_{s\in E} a_s x^s$, where $E\subset \R_+$ is finite.\\
We denote by $A\lb x^{\R_+}\rb$ the set of formal power series with exponents in $\R_+$, that is, the set of formal power series $\sum_{s\in \R_+}a_s x^s$ such that
\begin{equation}\label{cond:ps}\forall r\in \R,\ \ \#\{s\in \R_+\mid a_s\neq 0 \text{ and } s\leq r\}\leq\infty\end{equation}
For $f\in A\lb x^{\R_+}\rb$,  $f=\sum_{s\in {\R_+}}a_s x^s$, we define its \emph{support}:
$$\Supp(f):=\{s\in \R_+\mid a_s\neq 0\}.$$
The set $A\lb x^{\R_+}\rb$ is a subset of the ring of Hahn series, that is, power series  whose support is well-ordered (see \cite{Rib} for example).

We easily check that the usual addition and multiplication of polynomials and power series are well defined on $A[x^{\R_+}]$ and $A\lb x^{\R_+}\rb$, and make them commutative rings. The only difficulty is to check that the product of two power series is well defined, and this is where condition \eqref{cond:ps} is used.

For $f\in A\lb x^{\R_+}\rb$, $f\neq 0$, we define $\ord(f)$ to be the smallest $s\in \Supp(f)$. We define the \emph{initial monomial of $f$}:
$$\ini(f)=a_{\ord(f)} x^{\ord(f)}.$$
\end{definition}

Now we can state the following generalization of Theorem \ref{thm:Eis} that seems (along with its Corollary \ref{cor:eisenstein_several}) to cover all the different known generalizations of Eisenstein's Theorem:
\begin{theorem}[Eisenstein's Theorem for generalized series]\label{thm:eisenstein_general}
Let $A$ be an integral domain, $\K$ be its fraction field, and $x$ be one single indeterminate. Let $f(x)\in \K\lb x^{\R_+}\rb$ be algebraic and separable over $A\lb x^{\R_+}\rb$ that is written
$\displaystyle f=\sum_{s\in \R_+}a_sx^s.$
 Then there is $a\in A$ such that
$$  \forall s\in \R_+,\ \ \ a^{\lceil s\rceil+1}a_s\in A.$$
\end{theorem}

\begin{proof}
Let $P(x,y)\in A\lb x^{\R_+}\rb[y]$ be an irreducible polynomial with $P(x,f(x))=0$. Since  $f$ is separable over $A\lb x^{\R_+}\rb$, $\frac{\partial P}{\partial y}(x,f(x))\neq 0$.
 We set $e:=\ord\left(\frac{\partial P}{\partial y}(x,f(x))\right)$. We write
 $$f(x)=\sum_{\a\in\R^+}a_s x^s,\ f^{(0)}(x)=\sum_{s\leq e+1}a_s x^\a\text{ and } f^{(1)}(x)=f(x)-f^{(0)}(x).$$
By \ref{cond:ps}, $f^{(0)}(x)$ is a finite sum of monomials, hence  there exists $a\in A$ such that $af^{(0)}(x)\in A[x^{\R_+}]$. If $d=\deg_y(P(x,y))$, then $\ovl P(x,y):=a^dP(x,\frac{1}{a}y)\in A\lb x^{\R_+}\rb[y]$ and $\ovl P(x,af(x))=0$. So we can replace $P$ by $\ovl P$, and $f(x)$ by $af(x)$. Thus we way assume that $f^{(0)}(x)\in A[x^{\R_+}]$.

We have $P(x,f^{(0)}(x)+y)\in A\lb x^{\R_+}\rb[y]$ and $P(x,f^{(0)}(x)+f^{(1)}(x))=0$. The proposition is true for $f(x)$ if and only if it is true for $f^{(1)}(x)$. So we may replace $P(x,y)$ by $P(x,f^{(0)}(x)+y)$ and assume that $f^{(0)}(x)=0$. In particular   $\ord(f(x))>e+1$. Moreover, $\ord\left(\frac{\partial P}{\partial y}(x,f(x))\right)=e$ does not change under this modification.

We have
\begin{equation}\label{taylor}0=P(x,f(x))=P(x,0)+\frac{\partial P}{\partial y}(x,0)f(x)+f(x)^2R(x)\end{equation}
for some $R(x)\in A\lb x^{\R_+}\rb$. Because $\ord(f(x))>e$ and $\ord\left(\frac{\partial P}{\partial y}(x,f(x))\right)=e$, we have $\ord\left(\frac{\partial P}{\partial y}(x,0)\right)=e$. So $\ord(f(x)^2)>e=\ord\left(\frac{\partial P}{\partial y}(x,0)f(x)\right)$, and 
\begin{equation}\label{taylor:ini}\ini\left(\frac{\partial P}{\partial y}(x,0)f(x)\right)=-\ini(P(x,0)).\end{equation}
 In particular $\ord(P(x,0))> 2e$. We set $ax^e=\ini\left(\frac{\partial P}{\partial y}(x,0)\right)$.  We replace the indeterminate $y$ by $az$ where $z$ is a new indeterminate. We set $\ovl P(x,z)=\frac{1}{a^2x^e}P(x,az)$. We have
 \begin{equation}\label{eq:taylor2} \ovl P(x,z)=\frac{1}{a^2x^e}P(x,az)=\frac{P(x,0)}{a^2x^e}+\frac{1}{a x^e}\frac{\partial P}{\partial y}(x,0)z+z^2\ovl R(x,z)\end{equation}
 for some $\ovl R(x,z)\in A\lb x^{\R_+}\rb[z]$. We have $\ovl P(x,g(x))=0$ where $g(x)=\fract{f(x)}{a}$.
 
 Let us write $g(x)=\sum_{n\in\N}b_{s_n}x^{s_n}$ where $(s_n)_n$ is  increasing and $b_{s_n}\neq 0$ for every $n$. Since $\ovl P(x,g(x))=0$ and $\ini\left(\frac{1}{ax^e}\frac{\partial P}{\partial y}(x,0))\right)=1$, the coefficient of $x^{{s_n}}$ in   $\ovl P(x,g(x))$ is zero, and this one can be written as
 \begin{equation}\label{recursion}b_{s_n}-Q_n(b_{s_k})=0\end{equation}
 where $Q_n$ is a polynomial depending on some $b_{s_k}$ for $k<n$ with coefficients in $\frac{1}{a^2}A$.  Therefore, by induction on $n$, we have that $b_{s_n}=\frac{c_n}{a^{\ell(n)}}$ for some $c_n\in A$ and $\ell(n)\in\N$.
 
 Moreover $Q_n$ is a sum of monomials of the form $c_{s_{k_0}} b_{s_{k_1}}\cdots b_{s_{k_r}}$ with $r\leq d$ and $c_{s_{k_0}}$ is the coefficient of $x^{s_{k_0}}$ in the expansion of one of the coefficients of $\ovl P(x,z)$, and $s_{k_0}+s_{k_1}+\cdots+s_{k_r}=s_{n}$. 
 The coefficients of $\ovl R(x,z)$ are in $A$. We set
 $$s_{\min}:=\min\left\{\frac12\ord\left(\frac{P(x,0)}{a^2x^e}\right),\ord\left(\frac{1}{a x^e}\frac{\partial P}{\partial y}(x,0)-1\right)\right\}.$$
 We set $\la:=\frac{1}{ s_{\min}}$.
 Then for every monomial $cx^s$ in the expansion of $\frac{P(x,0)}{a^2x^e}$ or $\frac{1}{a x^e}\frac{\partial P}{\partial y}(x,0)-1$, we claim that $c=\frac{c'}{a^{m(s)}}$ with $m(s)\leq \la s$. Indeed, if $cx^s$ is a monomial of $\frac{P(x,0)}{a^2x^e}$, then we can choose $m(s)=2$; but $2=\frac{1}{s_{\min}}2s_{\min}\leq \la s_k$. The other case is similar.
 
 We prove by induction on $n$ that $\ell(n)\leq \la s_n$. Indeed, by induction, every monomial of the form  $c_{s_{k_0}} b_{s_{k_1}}\cdots b_{s_{k_r}}$ as before can be written as $\frac{c''}{a^{\la s_{k_0}+\ell({k_1})+\cdots+\ell({k_r})}}$ with $c''\in A$. Thus 
$$ \la {k_0}+\ell({k_1})+\cdots+\ell({k_r})\leq \la (s_{k_0}+s_{k_1}+\cdots+s_{k_r})=\la s_n.$$
So $\ell(n)\leq \la s_n$.
This proves the result.
 \end{proof}
 
 We can deduce the following version of Eisenstein's Theorem for power series in several indeterminates (this statement has first been proved over $\C$ in \cite{Saf}, and in positive characteristic when $n=1$ in \cite{BBC}):
 \begin{corollary}[Eisenstein's Theorem in several indeterminates]\label{cor:eisenstein_several}
Let $A$ be an integral domain, $\K$ be its fraction field, and $x=(x_1,\ldots,x_n)$ be a $n$-uple of indeterminates. Let $f(x)\in \K\lb x\rb$ be algebraic  over $A\lb x\rb$. Then there is $a\in A$ such that
$$af(ax_1,\ldots, ax_n)\in A\lb x\rb.$$
 \end{corollary}
 
 \begin{proof}
 Let $\o_1$, \ldots, $\o_n>0$ be linearly independent over $\Q$, and assume that $1\leq |\o_i|\leq 2$ for every $i$. Let $t$ be a new single indeterminate. We define a morphism
 $$\phi:\K\lb x_1,\ldots, x_n\rb\lgw \K\lb t^{\R_+}\rb$$
by $\phi(h(x))=h(t^{\o_1},\ldots, t^{\o_n})$. The images under $\phi$ of two different monomials are two different powers of $t$, since the $\o_i$ are $\Q$-linearly independent. Therefore $\phi$ is injective. Moreover $\phi(h)\in A\lb t^{\R_+}\rb$ if and only if $h\in A\lb x\rb$.\\
The ring of algebraic power series $\K\langle x\rangle$ is the henselization of the local ring $\K[x]_{(x)}$, thus $f\in\K\lb x\rb$, which is algebraic over $\K[x]$, is separable over $\K[x]$ (see \cite[p. 180]{Nag} for example). Therefore, $\phi(f)$ is algebraic and separable over $\K[t^{\R_+}]$. Let $f=\sum_{\a\in \N^n} a_\a x^\a$. Then $\phi(f)=\sum_{\a\in\N^n} a_\a t^{\langle \a,\o\rangle}$. 
By Theorem \ref{thm:eisenstein_general}, there is $b\in A$ such that $b^{\lceil \langle\a,\o\rangle\rceil}a_\a\in A$ for every $\a$. But $\langle \a,\o\rangle\leq 2|\a|$, so $b^{2|\a|}\in A$, hence, for   $a=b^2$, we have 
$$af(ax_1,\ldots, ax_n)\in A\lb x\rb.$$
 \end{proof}
 
 \begin{remark}
 We can relax the hypothesis of the previous corollary as follows: we keep the notations  of Corollary \ref{cor:eisenstein_several} and we assume now that $f(x)\in\L\lb x\rb$ where $\K\lgw \L$ is a field extension and that $f(x)$ is algebraic over $A[x]$. 
 Assume that $\L$ is separable over $\K$, or that $\K$ is a finitely generated extension field of a perfect field. Then the coefficients of $f(x)$ belong to a finite field extension $\K'$ of $\K$ (see \cite{CK}). Thus  we may assume that $\K'=\K(\alpha_1,\ldots, \alpha_r)$ for some $\alpha_i$ algebraic over $\K$. We set $A'=A[\alpha_1,\ldots, \alpha_r]$. Then, by the previous result, there is a $a'\in A'$ such that $a'f(a'x_1,\ldots, a'x_n)\in A'\lb x\rb$. Since $a'$ is algebraic over $A$, there is a nonzero $a''\in A'$ such that $a=a'a''\in A$. Moreover $af(ax_1,\ldots, ax_n)\in A'\lb x\rb$.
 \end{remark}

\section{Extensions, variations and applications}
\subsection{Effective bounds}
Assume here that $f(x)\in\ovl \Q\lb x\rb$ is algebraic over $\ovl \Q[x]$, where $x$ is a single indeterminate. Let us write $f(x)=\sum_{\ell} f_\ell x^\ell$. Let $\K$ be a number field containing the $a_\ell$. Let $N$ denotes the set of valuations of $\K$ extending the $p$-adic valuations of $\Q$. Then
Eisenstein's Theorem for $f(x)$ is equivalent to:
\begin{enumerate}
\item for all but finitely many  $\nu\in N$, $\nu(f_\ell)\geq 0$.
\item for every $\nu\in N$, there is $\lambda_\nu\in \R_+$ such that  $\nu(f_\ell)\leq \lambda_\nu(\ell+1)$.
\end{enumerate}
In this situation, a lot of work has been done to find sharp estimates of the lowest $\lambda_k$ satisfying the previous condition. We do not want to come up to this topic in details here, but the reader may consult \cite{Schm, DR, DcdP, BB} for further results in this direction.

\subsection{Analogue of Eisenstein's Theorem for D-finite power series}

\begin{definition}
Let $x$ be a single indeterminate.
A formal power series $f\in \Q\lb x\rb$  is called \emph{D-finite} if it satisfies a non trivial relation of the following form:
$$a_d(x)f^{(d)}(x)+\cdots+a_{1}(x)f'(x)+a_0(x)f(x)=0$$
where the $a_i$ are polynomials.
\end{definition}

\begin{example}
This definition is equivalent to saying that the $\Q(x)$-vector space generated by the derivatives of $f$ is finitely generated.\\
If $f$ is an algebraic power series, then $f$ satisfies a non trivial relation 
$$a_0(x)f^d(x)+\cdots+a_d(x)=0$$ where the $a_i(x)\in\Q(x)$ and $d$ is assumed to be minimal. By differentiating this equality, because $d$ is minimal and $\Q$ has characteristic zero, we obtain that $f'$ belongs to $\K(x,f(x))$ which is a finite extension of $\Q(x)$. By induction, all the derivatives of $f$ belong to $\Q(x,f(x))$. Therefore the $\Q(x)$-vector space generated by the derivatives of $f$ is a subspace of $\Q(x,f(x))$ that is finitely dimensional since $f$ is algebraic over $\Q(x)$. Therefore $f$ is $D$-finite.
\end{example}

\begin{remark}
In general, one defines $D$-finite power series with coefficients in a number field $\k$ to be a power series $f\in\k\lb x\rb$ such that the $\k(x)$-vector space generated by the derivatives of $f$ is finitely generated. But, for $f\in\k\lb x\rb$, if we write $f=\o_1f_1+\cdots+\o_sf_s$ where $(\o_1,\ldots, \o_s)$ is a $\Q$-basis of $\k$, then $f$ is $D$-finite if and only if the $f_i$ are $D$-finite (see \cite[Proposition 2.1]{DGS} for example). Therefore the study of $D$-finite power series with coefficients in a number field often reduces to the study of $D$-finite power series with rational coefficients.
\end{remark}

\begin{example}
The series
$$\exp(x)=\sum_{\ell\in\N} \frac{1}{\ell!}x^\ell,\ \ \ln(x)=\sum_{\ell\in\N}\frac{(-1)^{k+1}}{\ell}x^\ell$$
are $D$-finite. They are  not algebraic according to Eisenstein's Theorem.
\end{example}

The following result is an analogue of Eisenstein's Theorem for $D$-finite power series. It provides a bound on the number of primes dividing the denominators of the coefficients of a $D$-finite power series.

\begin{theorem}\label{thm:D-eisenstein}
Let $x$ be a single indeterminate.
Let $f(x)\in \Q\lb x\rb$ be a $D$-finite power series. Let us write $f(x)=\sum_{\ell\in\N} \frac{a_\ell}{b_\ell} x^\ell$ where the $a_\ell$ and $b_\ell$ are in $\Z$, and $\gcd(a_\ell, b_\ell)=1$. Then 
let us expand
$$\lcm\{b_0, \ldots, b_\ell\}=p_{\ell,1}\cdots p_{\ell, s_\ell}$$
 where the $p_{\ell, i}$ are prime elements. Then there is a contant $K>0$ such that
 $$\forall \ell\geq 2, \ \ s_\ell\leq K\ell\ln(\ell).$$
\end{theorem}

\begin{proof}
We set $f_\ell=\frac{a_\ell}{b_\ell}$. It is well known that $f$ is $D$-finite if and only if the sequence $(f_\ell)_\ell$ is $P$-recursive (see \cite[Theorem 1.5]{Sta} for example). This means that there are polynomials $p_0$, $p_1$, \ldots , $p_e\in \Z[t]$ such that
$$\forall \ell\in\N,\ \ p_0(\ell)f_\ell+p_1(\ell)f_{\ell+1}+\cdots+p_e(\ell)f_{\ell+e}=0.$$
Let $N\in\N$ large enough in order to insure that $p_e(\ell)\neq 0$ for $\ell\geq N$. Let $\ell\geq N+e$. We have
$$f_\ell=-\frac{p_0(\ell-e)f_{\ell-e}+p_1(\ell-e)f_{\ell-e+1}+\cdots+p_{e-1}(\ell-e)f_{\ell-1}}{p_e(\ell-e)}.$$
So $f_\ell=\dfrac{a}{\lcm\{b_{\ell-e},\ldots, b_{\ell-1}\}p_e(\ell-e)}$
 for some $a\in\Z$. There is a positive constant $C$ and $d\in\N$ such that  $|p_e(\ell-e)|\leq C\ell^d$ for all $\ell$ large enough. So the number of primes in the decomposition of $p_e(\ell-e)$ into primes is less than $\log_2(C\ell^d)$
 Therefore, we have
 $$s_{\ell+1}\leq s_\ell+\log_2(C\ell^d).$$
 So we have, since $x\log_2(Cx^d)-\dfrac{d}{\ln(2)}x$ is a primitive of $\log_2(Cx^d)$:
 $$s_{\ell+1}=\sum_{k=0}^{\ell}(s_{k+1}-s_k)+s_0\leq \sum_{k=0}^\ell \log_2(Ck^d)+s_0\leq C(\ell+1)\log_2(C(\ell+1)^d)-\frac{d(\ell+1)}{\log_2}+C'$$
 for some contant $C'$. This proves the statement.
\end{proof}

\begin{remark}
For $n\in\N$ and $k\leq n$,  the number of  integers  less than $n$ that are divisible by $k$ is $\left\lfloor \frac{n}{k}\right\rfloor$. Therefore $n!$ is the product of at least $\sum_{k=1}^n\left\lfloor \frac{n}{k}\right\rfloor\sim n\ln(n)$ prime numbers (counted with multiplicity). Since $\sum_{n\in\N}\frac{1}{n!}x^n$ is $D$-finite, the bound of Theorem \ref{thm:D-eisenstein} is asymptotically sharp.
\end{remark}

\begin{remark}
For a $D$-finite power series $f(x)\sum_{\ell\in\N} \frac{a_\ell}{b_\ell} x^\ell$ one can prove using the same kind of proof that, for a given prime $p$, the $p$-adic valuation of $\frac{a_\ell}{b_\ell} $ is bounded from below by a linear function (see \cite[Theorem 3]{Cla}).
\end{remark}
\subsection{Eisenstein series and local analytic geometry}
Regarding Corollary \ref{cor:eisenstein_several} this is natural to introduce the following notion:

\begin{definition}
Let $A$ be an integral domain, $\K$ be its fraction field, and let $n\in\N^*$. A power series $f(x)\in\K\lb x_1,\ldots, x_n\rb$ is called:
\begin{enumerate}
\item \emph{an Eisenstein series} if  there is $a\in A$ such that $af(ax)\in A\lb x\rb$,
\item \emph{a weakly Eisenstein series} if there is $a\in A$ such that $f(x)\in A_a\lb x\rb$.
\end{enumerate}
\end{definition}

\begin{remark}
The series $f(x)$ si weakly Eisenstein if $f(x)=\sum_{\a\in\N^n} \frac{a_\a}{a^{\b(\a)}}x^\a$ for some function $\b:\N^n\lgw \N$. It is Eisenstein if we can choose the expansion of $f(x)$ in such a way  that $\b(\a)=|\a|+1$. This last condition is equivalent to the fact that we can choose the expansion in  such a way that $\b$ is bounded by a linear function in $|\a|$. Indeed, if $\b(\a)\leq \lambda |\a|+\mu$ where $\lambda$, $\mu\in\N$, we can replace $a$ by $a^{\max\{\lambda,\mu,1\}}$, and for this new denominator, $\b(\a)=|\a|+1$.
\end{remark}

These two kinds of rings are very useful for several reasons:
\begin{enumerate}
\item They allow to work in families and to prove uniform results. More precisely, assume that by working with the ring of convergent power series $\C\{x_1,\ldots, x_n\}$ we are able to prove some local statement $S_0$ valid at the origin 0. Replacing the ring of convergent power series by the ring of (weakly) Eisenstein series where $A$ denotes some ring of functions depending on some indeterminates $\tau$, we may be able to prove the same statement for this new ring. Let $a$ be a common denominator of  the Eisenstein series appearing in the statement.
Then by specializing $\tau=\tau_0\in U$ with $a(\tau_0)\neq 0$ and $U$ some neighborhood of the origin, we recover the statement $S_{\tau_0}$ at the point $\tau_0$. If $U\setminus a^{-1}(0)$ is dense in $U$, this is a way of proving the statement $S_{\tau}$ uniformly in $\tau$ on some open dense subset of $U$. This kind of strategy is used in \cite{Zar, PP, BCR2}.
\item The rings of Eisenstein and weakly Eisenstein series satisfy the Weierstrass division and preparation Theorems, which are key tools in local analytic geometry (see Lemma \ref{lem:weierstrass} below).
\item These two rings not only contain the algebraic closure of $A\lb x\rb$ in $\K\lb x\rb$, but moreover they are algebraically closed in $\K\lb x\rb$.
\end{enumerate}

In practise there is not a big difference between the use of one definition rather than the other one.
 Zariski is apparently the first one to use this notion in \cite[pp. 467-469]{Zar}, in order to develop the theory of what is now called the \emph{Zariski equisingularity} for germs of analytic sets. In his work $A$ denotes the ring of polynomial in infinitely many indeterminates. Then Parusi\'nski and Paunescu \cite{PP} used the same notion to extend the work of Zariski, and in their case $A$ denotes the ring of polynomials $\k[\tau_1,\ldots, \tau_s]$ over a field $\k$. More precisely they used this notion to reduce the study of Zariski equisingularity for algebraic hypersurfaces on an open dense subset $U\subset \C^n$ to the study of Zariski equisingularity locally around a given point.  The very same kind of Eisenstein series has also been used in \cite{Ron2} to construct equisingular families of algebraic power series over $\ovl \Q$. The second definition appears in \cite{BCR2} where $A$ denotes a UFD (in particular some functions ring that is a UFD); and this definition is used to prove a uniform Gabrielov Rank Theorem \cite{BCR2}.

 \begin{lemma}[Weierstrass Preparation and Division Theorem]\label{lem:weierstrass}(cf. \cite[Lemma 3.6]{PP} or \cite[Prop. 4.9]{BCR2})
 
 Let $A$ be an integral domain, $\K$ be its fraction field,  and $f(x)$ and $g(x)\in \K\lb x_1,\ldots, x_n\rb$  be  Eisenstein series. Assume that $f(0,\ldots, 0, x_n)$ is a nonzero series of order $d$. Then:
 \begin{enumerate}
 \item We have 
 $$f(x)=\left(x_n^d+\sum_{i=1}^da_i(x_1,\ldots, x_{n-1})x_n^{d-i}\right)\times u(x)$$
 where $u(x)$ is a unit, that is,  $u(0)\neq 0$, and the $a_i$ are Eisenstein series.
 \item We can write in an unique way
 $$g(x)=f(x)q(x)+r(x)$$
 where $q(x)$ is an Eisenstein series and $r(x)$ is a polynomial in $x_n$ of degree $<d$ whose coefficients are Eisenstein series.
 \end{enumerate}
 The same statement is still true if we replace "Eisenstein series" by "weakly Eisenstein series".
 \end{lemma}


\subsection{Algebraic closure of the field of power series: Newton-Puiseux and MacDonald Theorems}
The aim of this part is to present some results of \cite{Ron1} and provide new and  simpler proofs of them, which are based on Theorem \ref{thm:eisenstein_general}.\\
Let $x=(x_1,\ldots, x_n)$ be a $n$-uple of indeterminates. \\
When $n=1$, the Newton-Puiseux Theorem asserts that the roots of any polynomial $P(x,y)\in\C(\!(x)\!)[y]$ are \emph{Puiseux series}. By definition a Puiseux series is just a series in $\C(\!(x^{\frac1q})\!)$ for some $q\in\N^*$, that is, a series of the following form:
$$a_{\ell_0}x^{\frac{\ell_0}{q}}+a_{\ell_0+1}x^{\frac{\ell_0+1}{q}}+a_{\ell_0+2}x^{\frac{\ell_0+2}{q}}+\cdots$$
where $\ell_0\in\Z$, and the $a_{\ell}$ belong to $\C$. On the other hand, every Puiseux series is algebraic over $\C(\!(x)\!)$ since  $\C(\!(x^{\frac1q})\!)\simeq\C(\!(x)\!)[x^{\frac1q}]$ is finite of degree $q$ over $\C(\!(x)\!)$. Therefore, an algebraic closure of $\C(\!(x)\!)$ is
$$\bigcup_{q\in\N^*}\C(\!(x^{\frac{1}{q}})\!).$$

When $n\geq 2$,  there is no general known description of an algebraic closure of $\C(\!(x)\!)$, elements algebraic over $\C(\!(x)\!)$ cannot be expanded as Puiseux series as $\sqrt{x_1+x_2}$. The main problem is that the field $\C(\!(x)\!)$ is no longer a complete field for the $(x)$-adic topology when $n\geq 2$. This last sentence requires some explanations:\\
The field $\C(\!(x)\!)$ can be equipped with the \emph{$(x)$-adic topology}. This one is induced by the norm defined by
$$\forall f,g\in\C\lb x\rb\setminus\{0\},\ \ \  \left|\frac{f}{g}\right|:=e^{-\ord(f/g)}$$
where $\ord(f)$ denotes the usual order of the power series $f$, and $\ord(f/g)=\ord(f)-\ord(g)$. Therefore, the higher is the order of $f$, the smaller is the norm of $f$. 

Now, the completion of $\C(\!(x)\!)$ with respect to the topology induced by $|\cdot|$, denoted by $\wdh{\C(\!(x)\!)}$, is the set of series written as
\begin{equation}\label{series-completion}\frac{f_{\ell_0}}{g_{\ell_0}}+\frac{f_{\ell_0+1}}{g_{\ell_0+1}}+\frac{f_{\ell_0+2}}{g_{\ell_0+2}}+\frac{f_{\ell_0+3}}{g_{\ell_0+3}}+\cdots\end{equation}
where $\frac{f_{\ell}}{g_{\ell}}=0$ or $\ord(f_{\ell}/g_{\ell})=\ell\in\Z$, and $\ell_0\in\Z$.
For a nonzero series $g$, we define $\ini(g)$ to be the \emph{initial term of $g$}, that is, the homogeneous polynomial of smallest degree in the expansion of $g$. Hence, $g=\ini(g)+h$ where $\ord(h)>\ord(g)$. So, we have
$$\frac{f}{g}=\frac{f}{\ini(g)}\times \frac{1}{1+\frac{h}{\ini(g)}}=\frac{f}{\ini(g)}\times \sum_{k=0}^\infty \left(-\frac{h}{\ini(g)}\right)^k=\frac{f}{\ini(g)}\times \sum_{k=0}^\infty \left(-\frac{h_0+h_1+h_2+\cdots}{\ini(g)}\right)^k$$
where  $h_d$ is the homogeneous polynomial of degree $d$ in the expansion of $h$.
Therefore, any series as in \eqref{series-completion} can be rewritten in such a way that  the $f_\ell$ and $g_\ell$ are homogeneous polynomials.\\
When $n=1$, then any homogeneous polynomial is a monomial, and each series as in \eqref{series-completion} is a Laurent series, that is, an element of $\C(\!(x)\!)$. When $n\geq 2$, this no longer the case: for example the series
$$\frac{x_1}{x_2}+\frac{x_1^2}{x_2^2}+\frac{x_1^6}{x_2^3}+\cdots+\frac{x_1^{\ell !}}{x_2^\ell}+\cdots\in \wdh{\C(\!(x)\!)}\setminus \C(\!(x)\!).$$
The fact that $\C(\!(x)\!)$ is not complete for $n\geq 2$, does not enable to use  the Newton method, on the contrary of the case $n=1$. The use of the Newton method imposes to be able to describe or to characterize the elements of $\wdh{\C(\!(x)\!)}$ that are algebraic over $\C(\!(x)\!)$.

We can generalize this definition by associating real positive numbers   $\o_i$ (called \emph{weights}) to the  indeterminates. If $\o_1$, \ldots, $\o_n>0$, we define
$$\nu_\o\left(\sum_{\a\in\N^n}f_\a x^\a\right):=\min\left\{\langle \o,\a\rangle\mid f_\a\neq 0\right\}$$
and $\nu_\o(f/g)=\nu_\o(f)-\nu_\o(g)$ for every power series $f$ and $g$. This $\nu_\o$ is a valuation and it induces an other norm:
$$\forall f,g\in\C\lb x\rb\setminus\{0\},\ \ \  \left|\frac{f}{g}\right|_\o:=e^{-\nu_\o(f/g)}.$$
\begin{remark}\label{expansion_completion}
We can write the elements of the completion of $\C(\!(x)\!)$ with respect to this norm, denoted by $\wdh{\C(\!(x)\!)}^\o$, as in \eqref{series-completion}, but the $f_\ell$ and $g_\ell$ are \emph{$\o$-weighted homogeneous} (a polynomial $f(x)$ is $\o$-weighted homogeneous if,  $\langle \o,\a\rangle=\langle \o,\b\rangle$ for all nonzero monomials $C_1x^\a$ et $C_2x^\b$ of $f(x)$). More precisely, the elements of $\wdh{\C(\!(x)\!)}^\o$ are written as
\begin{equation}\label{series-completion_omega}\frac{f_{\ell_0}}{g_{\ell_0}}+\frac{f_{\ell_1}}{g_{\ell_1}}+\frac{f_{\ell_2}}{g_{\ell_2}}+\frac{f_{\ell_3}}{g_{\ell_3}}+\cdots\end{equation}
where the $f_\ell$ and $g_\ell$ are $\o$-weighted homogeneous, $\nu_\o(f_\ell/g_\ell)=\ell$,  and $\lim_{\infty}\ell_k=\infty$.
\end{remark}

Then, the following result is proved in the very same way as Newton-Puiseux Theorem (see \cite[Prop. 3.28]{Ron1} for example):

\begin{theorem}[Generalized Newton-Puiseux Theorem]
Let $\o\in\R_{>0}^n$ with $\Q$-linearly independent coordinates. The roots of a  polynomial with coefficients in $\C(\!(x)\!)$ are Puiseux series of the form $\xi(x_1^{\frac1q},\ldots, x_n^{\frac1q})$ for some
for some $q\in\N^*$ and $\xi(x)\in\wdh{\C(\!(x)\!)}^\o$.
\end{theorem}
In the case $n=1$ we can choose $\o_1=1$ and we recover Newton-Puiseux Theorem since $\wdh{\C(\!(x)\!)}^\o=\C(\!(x)\!)$ in this case.\\
We can state an analogue of this result for any vector of weights $\o\in\R_{>0}^n$, but we need to introduce some technical material, and we will not do it here. What we need to remember, is that one can describe an algebraic closure of $\wdh{\C(\!(x)\!)}^\o$ (more easily in the case where the $\o_i$ are $\Q$-linearly independent). Therefore, a natural problem is to describe the algebraic closure of $\C(\!(x)\!)$ in $\wdh{\C(\!(x)\!)}^\o$, or characterize the elements of $\wdh{\C(\!(x)\!)}^\o$ that are algebraic over $\C(\!(x)\!)$. One answer is given by the following theorem (this is a particular case of \cite[Theorem 5.12]{Ron1}) that is an easy corollary of  Eisenstein's Theorem \ref{thm:eisenstein_general}:
\begin{theorem}\label{thm:completion}
Let $\o\in\R_{>0}^n$. Let $\xi\in\wdh{\C(\!(x)\!)}^\o$ be algebraic over $\C(\!(x)\!)$. Then there is a $\o$-weighted homogeneous polynomial $a(x)\in\C[x]$ such that
$$\xi=\sum_{\ell=\ell_0}^\infty\frac{a_\ell(x)}{a(x)^{\lceil\ell-\ell_0\rceil+1}}$$
where the $a_\ell(x)$ are $\o$-weighted homogenous polynomials and $\nu_\o\left(\frac{a_\ell(x)}{a(x)^{\lceil\ell-\ell_0\rceil+1}}\right)=\ell$.
\end{theorem}

\begin{proof}[Proof of Theorem \ref{thm:completion}]
Let us write $\xi(x)=\sum_{\ell=\ell_0}^\infty \frac{a_\ell(x)}{b_\ell(x)}$ where the $a_\ell(x)$ and $b_\ell(x)$ are $\o$-weighted homogeneous and the weight of $\frac{a_\ell(x)}{b_\ell(x)}$ equals $\ell$. In the expansion of $\xi(x)$, there is only finitely many $\ell$ with $\langle\o,\ell\rangle <0$ for which $\frac{a_\ell(x)}{b_\ell(x)}\neq 0$. Therefore, we only need to prove the result when $\ell_0\geq 0$. Indeed, then it is enough to multiply $a(x)$ by the $b_\ell(x)$ for $\ell<0$.

Assume that $P(x,\xi(x))=0$ where $P(x,y)\in\C\lb x\rb[y]$. Then 
$$P(t^\o x,\xi(t^\o x))=0$$ where
$t^\o x:=(t^{\o_1}x_1,\ldots, t^{\o_n}x_n)$
for a new indeterminate $t$. We set $Q(t,y)=P(t^\o x,y)\in\C[x]\lb t^{\R_+}\rb[y]$. By Remark \ref{expansion_completion}
$$\xi(t^\o x)=\sum_{\ell\geq 0}^\infty\frac{a_\ell(x)}{b_{\ell}(x)}t^{\ell}\in \C[x]\lb t^{\R_+}\rb.$$

 Let $A:=\C[x]$ and $\K=\C(x)$. 
By Theorem \ref{thm:eisenstein_general}, since $\xi(t^\o x)$ is algebraic and separable over $A\lb t^{\R_+}\rb$, there exists $a(x)\in\C[x]$ such that $b_\ell(x)$ divides $a(x)^{\lceil \ell\rceil+1}$. This proves the result.
\end{proof}


In the case where the $\o_i$ are $\Q$-linearly independent, we 
obtain the following  result (that extends a result of MacDonald \cite{McD} - see also \cite[Theorem 6.9]{Ron1}):

\begin{corollary}[MacDonald Theorem]
Let $\o\in\R^n_{>0}$ such that the $\o_i$ are $\Q$-linearly independent. Then the elements of $\wdh{\C(\!(x)\!)}^\o$ are Laurent series $\xi=\sum_{\a\in\Z^n}\xi_\a x^\a$ whose support:
$$\Supp(\xi):=\{\a\in\Z^n\mid \xi_\a\neq 0\}$$
verifies $\Supp(\xi)\subset \g+\si$ where $\g\in\Z^n$ and $\si$ is a rational strongly convex cone, that is,
$$\si:=\{\la_1s_1+\cdots+\la_r s_r\mid \la_i\in \R_{\geq 0}\}$$
where the $s_i\in\Z^n$, and $\si$ does not contain a line.
\end{corollary}

\begin{proof}
Let us remark that, when the $\o_i$ are $\Q$-linearly independent,  $\nu_\o(x^\a)=\nu_\o(x^\b)$ if and only of $\a=\b$. Therefore, the only $\o$-weighted homogeneous polynomials are the monomials. Thus, in Theorem \ref{thm:completion}, $a(x)$ is a monomial, let us say $a(x)=x^\b$.

Let $\xi\in \wdh{\C(\!(x)\!)}^\o$     be algebraic over $\C(\!(x)\!)$. By removing finitely many monomials from $\xi$, we may assume that $\xi=\sum_{\ell>0}\frac{a_\ell(x)}{x^{(\lceil\ell-\ell_0\rceil+1)\b}}$.

Let $\la\in\N^*$. The monomial map
$$\psi : (x_1,\ldots, x_n)\lgm (x_1x^{\la\o_1\b},\ldots, x_nx^{\la \o_n\b})$$
sends $\xi$ onto $\sum_{\ell>0} x^{(\la \ell-\lceil\ell-\ell_0\rceil+1)\b}a_\ell(x)$. So the support of $\psi_\la(\xi)$ is in $\R_{\geq 0}^n$ if $\la$ is large enough. The matrix of the monomial map $\psi_\la$ is $1\!\!1_n+\la \o\b^\perp$. Let $\chi(\la):=\det(1\!\!1_n+\la \o\b^\perp)$. Then $\chi(\la)$ is a polynomial of degree $\leq n$, and $\chi(0)\neq0$. Thus $\chi(\la)\neq 0$ for $\la\in\N$ large enough. So the monomial map $\psi_\la$ is invertible, and its inverse is a monomial map whose matrix is $(1\!\!1_n+\la \o\b^\perp)^{-1}$. Therefore, the support of $\xi$ is included in the image $C$ of $\R_{\geq 0}^n$ under a linear map defined over $\Q(\o_1,\ldots, \o_n)$. Thus $C$ is a polyhedral strongly convex cone. We can embed $C$ is a slightly larger rational polyhedral cone that is strongly convex. This proves the result.
\end{proof}

\subsection{Jung-Abhyankar Theorem and a question of Teissier}
When $n\geq 2$, there is a case where the roots of a polynomial with coefficients in $\C(\!(x)\!)$ can be expanded as Puiseux series. This is the following result:
\begin{theorem}[Jung-Abhyankar Theorem]
Let $x=(x_1,\ldots, x_n)$ and $y$ be one single indeterminate, and let $P(x,y)\in\C\lb x\rb[y]$ be monic in $y$. Assume that the discriminant of $P$ with respect to $y$ is of the form $x^\a u(x)$ where $u(0)\neq 0$. Then the roots of $P$ are in $\C\lb x^{\frac1q}\rb$ for some $q\in\N^*$ (here $x^{\frac1q}:=(x_1^{\frac1q},\ldots, x_n^{\frac1q})$).
\end{theorem}
There are essentially two kind of proofs of this result: proofs based on the analytic case where we can use topological methods (the case where the coefficients of $P$ are convergent - see \cite{Jun, PR}) and purely algebraic proofs based on ramification theory (see \cite{Abh, KV}).

One consequence of the Jung-Abhyankar Theorem is the following fact: \emph{if the discriminant $\Delta_P(x)$ of a monic polynomial $P(x,y)$ is of the form $x^\a u(x)$, where $u(0)\neq 0$, then the irreducible factors of $P(x,y)$ in $\C(\!( x)\!)[y]$ remain irreducible in $\wdh\C(\!(x)\!)[y]$}. 

We remark that, when  the $\o_i$ are $\Q$-linearly independent, a polynomial is $\o$-weighted homogeneous if and only if it is a monomial. Therefore,
the previous consequence  can be rephrased in: 

\emph{Let $\o\in\R_{>0}^n$ such that the $\o_i$ are $\Q$-linearly independent;  if  $\Delta_P(x)=q(x)u(x)$ where $u(0)\neq 0$ and $q(x)$ is $\o$-weighted homogeneous , then the irreducible factors of $P(x,y)$ in $\C(\!( x)\!)[y]$ remain irreducible in $\wdh\C(\!(x)\!)[y]$. }

B. Teissier asked if it was possible to prove this last statement without using Jung-Abhyankar Theorem.  The following result is a generalization of this last statement, whose proof does not invoque Jung-Abhyankar Theorem but is  based on a topological argument:

\begin{theorem}\cite[Theorem 7.5]{Ron1}
Let $x=(x_1,\ldots, x_n)$ and $y$ be one single indeterminate. Let $\o\in\R^n_{>0}$. Let $P\in\C\lb x\rb[y]$ be a monic polynomial in $y$ such that $\Delta_P(x)=\delta(x) u(x)$ where $\delta(x)$ is $\o$-weighted homogeneous and $u(x)\neq 0$. Then the monic irreducible factors of $P(y)$ in $\C\lb x\rb[y]$ remain irreducible in $\wdh{\C(\!(x)\!)}^\o[y]$.
\end{theorem}

\begin{proof}[Sketch of proof]
First of all we assume that the $\o_i$ are positive integers and the coefficients of $P$ are convergent power series. Let $B(\rho)\subset \C^n$ denote the open ball of radius $\rho>0$ centered at the origin. Let $Q(y)\in\wdh{\C(\!(x)\!)}^\o[y]$ be an irreducible factor of $P$, monic in $y$. Let $A$ be a coefficient of $Q$.
By Theorem \ref{thm:completion}, we can write 
\begin{equation}\label{eq_proof}A=\sum_{\ell=\ell_0}^\infty\frac{a_\ell(x)}{a(x)^{\lceil\ell-\ell_0\rceil+1}}\end{equation}
where the $a_\ell(x)$ are $\o$-weighted homogenous polynomials and $\nu_\o\left(\frac{a_\ell(x)}{a(x)^{\lceil\ell-\ell_0\rceil+1}}\right)=\ell$.

Since the coefficients of $P$ are convergent, one can prove a "convergent version" of Theorem \ref{thm:completion} (cf. \cite[Theorem 6.18]{Ron1}), in the sense that $A$ is convergent if $|x|$ is small enough while $|a(x)|$ is not too small. More precisely,
for $z$, $z'\in\C^n$, we set $\|z\|_\o:=\max_i|z_i|^{1/\o_i}$, and $d_\o(z,z'):=\|z-z'\|_\o$. Then, there is $\rho>0$ and $C>0$ such that $A$ is convergent on any following open set
$$U_{\rho, K,\varepsilon}:=\left\{z\in B(\rho)\mid \|z\|_\o<\varepsilon \text{ and } d_\o(z,a^{-1}(0))>K\|z\|_\o\right\}$$
with $\frac{\varepsilon}{K^\o}<C$.

On the other hand, the roots of $P$ are local analytic functions on a neighborhood  of the origin outside the discriminant locus of $P$, that is, on  $B(\rho')\setminus\{\delta^{-1}(0)\}$ for some $\rho'>0$. Here  \emph{local analytic} means that, for any $p\in B(\rho')\setminus\{\delta^{-1}(0)\}$ the roots of $P$ are well defined analytic functions on some neighborhood of $p$, but these roots may not be extended in well defined functions on $B(\rho')\setminus\{\delta^{-1}(0)\}$ since there may be some monodromy.
 By replacing $\rho$ by the minimum of $\rho$ and $\rho'$ we assume $\rho'=\rho$. Since 
$$A=\sum_{\ell=\ell_0}^\infty\frac{a_\ell(x)}{a(x)^{\lceil\ell-\ell_0\rceil+1}}=\sum_{\ell=\ell_0}^\infty\frac{a_\ell(x)\delta(x)^{\lceil\ell-\ell_0\rceil+1}}{(\delta(x)a(x))^{\lceil\ell-\ell_0\rceil+1}}$$
we can replace $a(x)$ by $\delta(x)a(x)$ and assume that the roots of $P$ are local analytic functions on $B(\rho)\setminus\{a^{-1}(0)\}$. Since $A$ is a symmetric polynomial depending on the roots of $P$, $A$ is local analytic on $B(\rho)\setminus\{a^{-1}(0)\}$.

One can prove that the union of the $U_{\rho,K,\varepsilon}$ where $K$ and $\varepsilon$ run over the positive real numbers such that $\frac{\varepsilon}{K^\o}<C$ is a domain $V$ that has the same homotopy type as $B(\rho)\setminus\{a^{-1}(0)\}$. Therefore, $A$, which is analytic on $V$ and local analytic on $B(\rho)\setminus\{a^{-1}(0)\}$, is analytic on $B(\rho)\setminus\{a^{-1}(0)\}$. Moreover, since $P$ is monic in $y$, the roots of $P$ are locally bounded around the origin. Thus $A$ is locally bounded around the origin. Then, by Riemann Removable Singularity Theorem, $A$ extends to an analytic function in a neighborhood of the origin. This proves that the coefficients of $Q$ are analytic in a neighborhood of the origin, thus are convergent power series in $x$. So $Q$ is indeed an irreducible factor of $P$ in $\C\lb x\rb[y]$.

Now assume that the coefficients of $P$ are convergent power series, but the $\o_i$ are any positive numbers. Once again we consider an irreducible factor $Q\in\wdh{\C(\!(x)\!)}^\o[y]$  of $P$, and $A$ one its coefficients. We can expand $A$ as in \eqref{eq_proof}, and once again we may assume that $\delta(x)$ divides $a(x)$. Now the idea is that, if $\a(x)$ is $\o$-weighted homogeneous, then it is again $\o'$-weighted homogeneous when $\o'$ is close enough to $\o$ and the $\o'_i$ satisfy all the $\Q$-linear relations satisfied by the $\o_i$ (and possibly new ones). In particular we can choose $\o'\in\Q_{>0}^n$. Thus, by replacing $\o$ by $\o'$,  we are reduced to the previous case, and the coefficients of $Q$ are once again convergent power series.

Finally, when the coefficients of $P$ are no longer convergent, we may reduce to the convergent case by approching $P$ by a new polynomial $P'$ whose coefficients are analytic, $\Delta_{P'}$ is also a $\o$-weighted homogeneous polynomial times a unit, and the irreducible factors of $P'$ approximate those  of $P$. This step is mainly based on Artin Approximation Theorem. The reader may refer to \cite{Ron1} for the details.
\end{proof}

\begin{question}
Let $\nu$ be a rank one valuation of $\C\lb x\rb$ centered at $(x)$. This means that  $\nu\left(\C\lb x\rb\setminus\{0\}\right)$ is a subset of $\R_+$ and that $\{f\in\C\lb x\rb\mid \nu(f)>0\}=(x)$. We denote by $V_\nu$ the valuation of $\nu$, and by $\wdh V_\nu$ its completion. Let $P\in\C\lb x\rb[y]$ be monic in $y$. A natural question is to find a condition on $\Delta_P$ to insure that the irreducible factors of $P$ in $\C\lb x\rb[y]$ remain irreducible in $\wdh V_\nu[y]$.
\end{question}



\end{document}